\documentclass[12pt]{article}

\input diagram.tex

\usepackage{amsmath} 
\usepackage[mathscr]{eucal}
\usepackage{mathrsfs}
\usepackage{amssymb} 
\usepackage{theorem}
\usepackage{rotating}
\usepackage{stmaryrd}

\theoremstyle{change} 
\newtheorem{theorem}{Theorem}[section] 
\newtheorem{proposition}[theorem]{Proposition}
\newtheorem{corollary}[theorem]{Corollary}
\theorembodyfont{\rmfamily} 

\newtheorem{remark}[theorem]{Remark}
\newtheorem{example}[theorem]{Example}

\newtheorem{definition}[theorem]{Definition}
\newtheorem{notation}[theorem]{Notation}

\newtheorem{remark and notation}[theorem]{Remark and Notation}

\newenvironment{proof}{\noindent{\bf Proof}\ }{\qed\bigskip}

\renewcommand{\le}{\leqslant}
\renewcommand{\ge}{\geqslant} 

\newcommand{\Aut}{\mathrm{Aut}}

\newcommand{\BP}{\mathcal{BP}}
\newcommand{\Br}{\mathrm{Br}}

\newcommand{\btilde}{\tilde{b}}

\newcommand{\calC}{\mathcal{C}}
\newcommand{\calCtilde}{\tilde{\calC}}

\newcommand{\calF}{\mathcal{F}}
\newcommand{\calFtilde}{\tilde{\calF}}

\newcommand{\End}{\mathrm{End}}
\newcommand{\etilde}{\tilde{e}}
\newcommand{\FF}{\mathbb{F}}
\newcommand{\ftilde}{\tilde{f}}

\newcommand{\Gal}{\mathrm{Gal}}

\newcommand{\Hom}{\mathrm{Hom}}

\newcommand{\lexp}[2]{\setbox0=\hbox{$#2$} \setbox1=\vbox to
                 \ht0{}\,\box1^{#1}\!#2}

\newcommand{\myiso}{\buildrel\sim\over\to}

\newcommand{\psihat}{\hat{\psi}}

\newcommand{\qed}{\nobreak\hfill
                   \vbox{\hrule\hbox{\vrule\hbox to 5pt
                   {\vbox to 8pt{\vfil}\hfil}\vrule}\hrule}}

\newcommand{\stab}{\mathrm{stab}}

\newcommand{\Tr}{\mathrm{Tr}}
\newcommand{\varphihat}{\hat{\varphi}}

\newcommand{\ZZ}{\mathbb{Z}}

\title{Fusion systems of blocks of finite groups over arbitrary fields\footnote{{\bf MR Subject Classification:}  
20C20. {\bf Keywords:}  Blocks of finite groups; fusion systems}}

\author{\small Robert Boltje, \c{C}isil Karag\"uzel, Deniz Y\i lmaz\\
  \small Department of Mathematics\\
  \small University of California\\
  \small Santa Cruz, CA 95064\\
  \small U.S.A.\\
  \small boltje@ucsc.edu, ckaraguz@ucsc.edu, deyilmaz@ucsc.edu}
\date{June 20, 2019}

\begin{document}
\sloppy


\maketitle


\begin{abstract}
\noindent
To any block idempotent $b$ of a group algebra $kG$ of a finite group $G$ over a field $k$ of characteristic $p>0$, Puig associated a fusion system and proved that it is saturated if the $k$-algebra $kC_G(P)e$ is split, where $(P,e)$ is a maximal $kGb$-Brauer pair. We investigate in the non-split case how far the fusion system is from being saturated by describing it in an explicit way as being generated by the fusion system of a related block idempotent over a larger field together with a single automorphism of the defect group.
\end{abstract}


\section{Introduction}\label{sec intro}
Let $k$ be a field of characteristic $p$, let $G$ be a finite group and let $b$ be a block idempotent of $kG$. Puig defined a fusion system $\calF_{(P,e)}(kGb)$ associated to $kGb$ after choosing a maximal $kGb$-Brauer pair $(P,e)$. Up to category isomorphism, this fusion system does not depend on the choice of $(P,e)$. Puig also proved that $\calF_{(P,e)}(kGb)$ is saturated if the $k$-algebra $kC_G(P)e$ is split. It is known that in the non-split case it can happen that the fusion system associated to $kGb$ is not saturated. In fact, the Sylow axiom can fail, while the extension axiom always holds. In the Main Theorem~\ref{thm A} of this paper we establish a precise connection between the fusion systems of related blocks in a Galois extension $L/K$ of fields of characteristic $p$ with Galois group $\Gamma$. More precisely, let $b$ be a block idempotent of $LG$ and $\btilde$ the unique block idempotent of $KG$ with $b\btilde=b$. Moreover, let $(P,e)$ be a maximal $LGb$-Brauer pair and let $\etilde$ be the unique block idempotent of $KC_G(P)$ with $e\etilde=e$. Then $(P,\etilde)$ is a maximal $KG\btilde$-Brauer pair and one has an inclusion of the fusion systems 
\begin{equation*}
  \calF:=\calF_{(P,e)}(LGb)\subseteq \calF_{(P,\etilde)}(KG\btilde)=:\calFtilde\,.
\end{equation*}
Theorem~\ref{thm A} states that there exists an element $\sigma\in\Aut_{\calFtilde}(P)$ such that $\calFtilde=\langle\calF,\sigma\rangle$. As consequences of the nature of $\sigma$ we obtain that $\calFtilde$ is saturated if and only if $\calF$ is saturated and $p$ does not divide the index $[\Gamma_b:\Gamma_e]=[K(e):K(b)]$ of the stabilizers of $b$ and $e$ under the Galois action, or equivalently the degree of the field extensions after adjoining the coefficients of $e$ and $b$ to $K$. In the case that $L$ is chosen such that $LC_G(P)e$ is split, this gives a particularly handy criterion for a fusion system of a block $KG\btilde$ in the non-split case to be saturated, see Theorem~\ref{thm C}. The main result allows an alternative easy proof for the known fact that the extension axiom holds also in the non-split case, see Theorem~\ref{thm B}. Finally, the Main Theorem implies that a weak form of Alperin's fusion theorem holds also for arbitrary block fusion systems, see Theorem~\ref{thm D}.

\begin{notation}
We will use the following standard notations without further notice:

\smallskip
For a group $G$ and $x\in G$, we denote by $c_x\colon G\to G$ the conjugation map $g\mapsto xgx^{-1}$. If $k$ is a commutative ring, its $k$-linear extension to the group algebra $kG$ is again denoted by $c_x\colon kG\to kG$. We frequently will use left-exponential notation  $\lexp{x}{(-)}:=c_x$ for these maps. The maps $c_x$, $x\in G$, define an action of $G$ on $kG$ via $k$-algebra homomorphisms.

\smallskip
For $H\le G$, we denote by $[G/H]$ a set of representatives of the cosets $G/H$.

\smallskip
If a group $G$ acts on a set $X$, we usually denote the stabilizer of an element $x\in X$ by $G_x$. Moreover, for $H\le G$, we denote by $X^H$ the set of $H$-fixed points of $X$.
\end{notation}


\section{Brauer pairs}\label{sec Brauer pairs for block algebras}

Throughout this section, $G$ denotes a finite group, $k$ denotes a field of characteristic $p>0$, and $b$ denotes a block idempotent of $kG$, i.e., a primitive idempotent of $Z(kG)$. We recall the definition and properties of Brauer pairs for $kG$ following the treatment in \cite[IV.2]{AKO2011}. We note that the blanket assumption in \cite[IV.2]{AKO2011} that $k$ is algebraically closed is not used in the proofs of any of the statements that we cite from there. Alternatively, see also \cite[Sections 5.9 and 6.3]{Linckelmann2018}. 

\bigskip
Recall that, for a $p$-subgroup $P$ of $G$, the {\em Brauer homomorphism} with respect to $P$ is the $k$-linear projection map $\Br_P\colon (kG)^P\to kC_G(P)$, $\sum_{g\in G}\alpha_g g\mapsto \sum_{g\in C_G(P)} \alpha_g g$. This is a surjective $k$-algebra homomorphism which respects $G$-conjugation: $c_x\circ\Br_P=\Br_{\lexp{x}{P}}\circ c_x\colon (kG)^P\to kC_G(\lexp{x}{P})$ for $x\in G$. Thus, $\Br_P(b)$ is an idempotent of $Z(kC_G(P))=(kC_G(P))^{C_G(P)}$. Recall further that a {\em $kG$-Brauer pair} is a pair $(P,e)$ consisting of a $p$-subgroup $P$ of $G$ and a block idempotent $e$ of $kC_G(P)$. If $e$ occurs in the unique decomposition of $\Br_P(b)$ into a sum of primitive idempotents of $Z(kC_G(P))$ (that is, if $\Br_P(b)e=e$), then we call $(P,e)$ a {\em $(kG,b)$-Brauer pair}. We denote by $\BP(kG)$ the set of $kG$-Brauer pairs and by $\BP(kG,b)$ the set of $(kG,b)$-Brauer pairs. Clearly, $\BP(kG)$ is the disjoint union of the subsets $\BP(kG,b)$, where $b$ runs through the block idempotents of $kG$. The set $\BP(kG)$ is a $G$-set under the conjugation action given by $\lexp{x}{(P,e)}:=(\lexp{x}{P},\lexp{x}{e})$, and the subset $\BP(kG,b)$ is $G$-stable. Finally, we say that an idempotent $i$ of $(kG)^P$ is {\em associated} to a $kG$-Brauer pair $(P,e)$ if
\begin{equation*}
  e \Br_P(i) = \Br_P(i)\neq 0\,.
\end{equation*}
Note that if $i$ is primitive in $(kG)^P$ then $e\Br_P(i)\neq 0$ implies that $\Br_P(i)\neq 0$ and that $\Br_P(i)$ is primitive in $kC_G(P)$. Thus, $e\Br_P(i)=\Br_P(i)$.
One writes $(Q,f)\le (P,e)$ if $Q\le P$ and if any primitive idempotent $i$ of $(kG)^P$ which is associated to $(P,e)$ is also associated to $(Q,f)$, see \cite[Definition~2.9]{AKO2011}. This relation has the following properties.

\begin{theorem}(\cite[Theorems~2.10, 2.16]{AKO2011})\label{thm Brauer pairs}
{\rm (a)} Let $(P,e)\in\BP(kG)$ and let $Q\le P$. Then there exists a unique block idempotent $f$ of $kC_G(Q)$ such that $(Q,f)\le (P,e)$.

\smallskip
{\rm (b)} Let $(Q,f)\le(P,e)$ be in $\BP(kG)$ with $Q\trianglelefteq P$. Then $f$ is the unique block idempotent of $kC_G(Q)$ which is $P$-stable and satisfies $Br_P(f)e=e$.

\smallskip
{\rm (c)} The relation $\le$ on $\BP(kG)$ is a partial order which is respected by the conjugation action of $G$.
\end{theorem}

Clearly $(\{1\},b)\in\BP(kG,b)$ and Part~(b) of the above theorem implies that if $(P,e)\in\BP(kG,b)$ then $(\{1\},b)\le(P,e)$. Parts~(a) and (c) further imply that if $(Q,f)\le(P,e)$ holds for elements in $\BP(kG)$ then $(Q,f)\in\BP(kG,b)$ if and only if $(P,e)\in\BP(kG,b)$.

\smallskip
For Brauer pairs $(Q,f), (P,e)\in\BP(kG)$ one writes $(Q,f)\trianglelefteq (P,e)$ if $Q\trianglelefteq P$, $f$ is $P$-stable and $\Br_P(f)e=e$, cf.~\cite[Definition~IV.2.13]{AKO2011}. The following result is well-known to specialists. We state it for convenient future reference and give a proof for the convenience of the reader.

\begin{theorem}\label{thm (Q,f)le(P,e) equivalences}
For $(Q,f),(P,e)\in \BP(kG)$ with $Q\le P$ the following statements are equivalent:

\smallskip
{\rm(i)} One has $(Q,f)\le (P,e)$.

\smallskip
{\rm (ii)} There exist primitive idempotents $i$ of $(kG)^P$ and $j$ of $(kG)^Q$ such that $ij=j=ji$, $\Br_P(i)e\neq 0$ and $\Br_Q(j)f\neq 0$.

\smallskip
{\rm (iii)} There exist Brauer pairs $(Q_i,d_i)\in\BP(kG)$, $i=0,\ldots,n$, such that
\begin{equation*}
  (Q,f)=(Q_0,d_0)\trianglelefteq(Q_1,d_1)\trianglelefteq\cdots\trianglelefteq(Q_n,d_n)=(P,e)\,.
\end{equation*}

\smallskip
{\rm (iv)} For every primitive idempotent $i$ of $(kG)^P$ with $Br_P(i)e\neq 0$ one has $Br_Q(i)f\neq 0$.

\smallskip
{\rm (v)} There exists a primitive idempotent $i$ of $(kG)^P$ such that $\Br_P(i)e\neq 0$ and $\Br_Q(i)f=\Br_Q(i)\neq 0$.

\smallskip
{\rm (vi)} There exists a primitive idempotent $i$ of $(kG)^P$ such that $Br_P(i)e\neq 0$ and $\Br_Q(i)f\neq 0$.
\end{theorem}

\begin{proof}
The equivalences (i)$\iff$(ii)$\iff$(iii) follow from \cite[Proposition IV.2.14]{AKO2011}. Moreover, the implications (i)$\Rightarrow$(iv) and (v)$\Rightarrow$(vi) are trivial and the implication (i)$\Rightarrow$(v) follows from the fact that the image of a primitive idempotent under a surjective $k$-algebra homomorphism is either $0$ or a primitive idempotent.

Next we show that (iv) implies (i). Let $i$ be a primitive idempotent of $(kG)^P$ such that $\Br_P(i)e=\Br_P(i)\neq 0$. By (iv), $\Br_Q(i)f\neq 0$. By Theorem~\ref{thm Brauer pairs}(a) there exists a block idempotent $f'$ of $kC_G(Q)$ such that $(Q,f')\le (P,e)$. Thus, $\Br_Q(i)f'=Br_Q(i)$ which implies that $0\neq \Br_Q(i)f=\Br_Q(i)f'f$ and further that $f=f'$ and thus $(Q,f)\le (P,e)$.


Finally, we show that (vi) implies (i). Let $i$ be as in (vi). By Theorem~\ref{thm Brauer pairs}(a) there exists a block idempotent $f'$ of $kC_G(Q)$ such that $(Q,f')\le (P,e)$. This implies $\Br_Q(i)f'=\Br_Q(i)\neq 0$ and $0\neq\Br_Q(i)f=\Br_Q(i)f'f$. Thus $f=f'$ and $(Q,f)\le(P,e)$.
\end{proof}

\medskip
Recall that if $I\le H\le G$ then we have a well-defined trace map
\begin{equation*}
  \Tr_I^H\colon (kG)^I\to (kG)^H\,, \quad a\mapsto\sum_{x\in[H/I]} \lexp{x}{a}\,.
\end{equation*}
A subgroup $P$ of $G$, minimal with the property that $b\in\Tr_P^G((kG)^P)$, is called a {\em defect group} of the block idempotent $b$ and of the block algebra $kGb$. The defect groups of $kGb$ form a  single $G$-conjugacy class of $p$-subgroups of $G$. Maximal elements in $\BP(kG,b)$ enjoy properties that resemble the Sylow Theorem for finite groups.

\begin{theorem}(\cite[Theorem~2.20]{AKO2011})\label{thm maximal Brauer pairs}
{\rm (a)} The maximal elements in $\BP(kG,b)$ with respect to $\le$ form a single $G$-orbit.

\smallskip
{\rm (b)} For $(P,e)\in \BP(kG,b)$ the following are equivalent.

\smallskip
\quad {\rm (i)} $(P,e)$ is a maximal element in $\BP(kG,b)$.

\smallskip
\quad {\rm (ii)} $P$ is a defect group of $kGb$.

\smallskip
\quad {\rm (iii)} $P$ is maximal among all $p$-subgroups of $G$ with the property $\Br_P(b)\neq 0$.
\end{theorem}


\section{Fusion systems of block algebras}\label{sec fusion systems}

Throughout this section, $p$ is a prime. We first recall the basic notions and properties of fusion systems, a structure introduced by Puig. Our terminology follows \cite[Chapter I]{AKO2011}.

For subgroups $Q$ and $R$ of a finite group $G$ we denote by $\Hom_G(Q,R)$ the set of all group homomorphisms $\varphi\colon Q\to R$ with the property that there exists $x\in G$ with $\varphi(u)=c_x(u)$ for all $u\in Q$. Moreover, we set $\Aut_G(Q):=\Hom_G(Q,Q)$.

\begin{definition}(\cite[Definition~I.2.1]{AKO2011})\label{def fusion system}
Let $P$ be a finite $p$-group. A subcategory $\calF$ of the category of finite groups whose objects are the subgroups of $P$ is called a {\em fusion system} over $P$ if for any two subgroups $Q$ and $R$ of $P$, the set $\Hom_\calF(Q,R)$ has the following properties:

\smallskip
(i) $\Hom_P(Q,R)\subseteq\Hom_\calF(Q,R)$ and every element of $\Hom_{\calF}(Q,R)$ is injective.

\smallskip
(ii) For each $\varphi\in\Hom_{\calF}(Q,R)$, the group isomorphism $Q\to\varphi(Q)$, $u\mapsto \varphi(u)$, and its inverse are morphisms in $\calF$.
\end{definition}

For instance, if $G$ is a finite group and $P$ is a $p$-subgroup of $G$, we obtain a fusion system $\calF_P(G)$ over $P$ by setting $\Hom_{\calF_P(G)}(Q,R):=\Hom_G(Q,R)$, for all subgroups $Q$ and $R$ of $P$. Note that the intersection of two fusion systems over $P$ is again a fusion system and that a fusion system over $P$ is determined by the isomorphisms it contains. Thus the smallest fusion system over a finite $p$-group $P$ is the fusion system $\calF_P(P)$.

\begin{definition}(\cite[Definition~I.2.4]{AKO2011})\label{def fully centralized}
Let $\calF$ be a fusion system over a finite $p$-group $P$. A subgroup $Q$ of $P$ is called {\em fully $\calF$-centralized} if $|C_P(Q)|\ge|C_P(Q')|$ for any subgroup $Q'$ of $P$ which is $\calF$-isomorphic to $Q$. Similarly, $Q$ is called {\em fully $\calF$-normalized} if $|N_P(Q)|\ge|N_P(Q')|$ for any subgroup $Q'$ of $P$ which is $\calF$-isomorphic to $Q$.
\end{definition}

\begin{definition}(\cite[Definition~I.2.2]{AKO2011})\label{def Nphi}
Let $\calF$ be a fusion system on a finite $p$-group $P$ and let $\varphi\colon Q\to R$ be an isomorphism in $\calF$. One denotes by $N_\varphi$ the set of all elements $y\in N_P(Q)$ for which there exists $z\in N_P(R)$ with the property $\varphi\circ c_y = c_z\circ\varphi\colon Q\to R$. Note that $QC_P(Q)\le N_\varphi\le N_P(Q)$ and that $N_\varphi$ does not depend on $\calF$, but only on $\varphi$ and $P$.
\end{definition}

If $\calF$ is a fusion system over a finite $p$-group $P$ and $Q\le P$ then we set $\Aut_\calF(Q):=\Hom_{\calF}(Q,Q)$, a subgroup of the automorphism group of $Q$. The following definition of saturation goes back to Stancu and is an equivalent reformulation of the original definition, see \cite[Proposition~I.9.3]{AKO2011}.

\begin{definition}\label{def saturated}
A fusion system $\calF$ over a $p$-group $P$ is called {\em saturated} if the following two conditions hold.

\smallskip
(i) {\em Sylow axiom:} The group $\Aut_P(P)$ is a Sylow $p$-subgroup of $\Aut_\calF(P)$.

\smallskip
(ii) {\em Extension axiom:} For every $Q\le P$ and every $\varphi\in\Hom_{\calF}(Q,P)$ such that $\varphi(Q)$ is fully $\calF$-normalized there exists a morphism $\psi\in\Hom_\calF(N_\varphi,P)$ whose restriction to $Q$ equals $\varphi$.
\end{definition}

For instance, if $P$ is a Sylow $p$-subgroup of a finite group $G$ then the fusion system $\calF_P(G)$ is saturated (see \cite[Theorem~2.3]{AKO2011}).

\begin{definition}\label{def block fusion system}
Let $G$ be a finite group, let $k$ be a field of characteristic $p$, let $b$ be a block idempotent of $kG$, and let $(P,e)$ be a maximal $(kG,b)$-Brauer pair. We define a category $\calF_{(P,e)}(kGb)$ as follows. First, for every $Q\le P$ denote by $e_Q$ the unique block idempotent of $kC_G(Q)$ with $(Q,e_Q)\le (P,e)$. The objects of $\calF_{(P,e)}(kGb)$ are the subgroups of $P$ and for subgroups $Q$ and $R$ of $P$ let $\Hom_{\calF_{(P,e)}(kGb)}(Q,R)$ denote the set of group homomorphisms $\varphi\colon Q\to R$ such that there exists $x\in G$ with $\varphi(u)=c_x(u)$ for all $u\in Q$ and $\lexp{x}{(Q,e_Q)}\le(R,e_R)$. Composition in $\calF_{(P,e)}(kGb)$ is the usual composition of functions.
\end{definition}

\begin{remark}\label{rem block fusion system}
Let $kG$, $b$, and $(P,e)$ be as in Definition~\ref{def block fusion system}.

\smallskip
(a) It is clear from the definition that $\calF_{(P,e)}(kGb)$ is a fusion system over $P$. 

\smallskip
(b) If $kGb$ is the principal block of $kG$, then by Brauer's third main theorem, $\calF_{(P,e)}(kGb)$ is equal to $\calF_P(G)$ and $P$ is a Sylow $p$-subgroup of $G$. Thus, $\calF_{(P,e)}(kGb)$ is saturated in this case.

\smallskip
(c) Example~\ref{ex Sylow axiom fails} below shows that in general the Sylow axiom does not hold for $\calF_{(P,e)}(kGb)$. But we will show in Theorem~\ref{thm B} that the extension axiom holds for $\calF_{(P,e)}(kGb)$.
\end{remark}

The following theorem was first proved by Puig. It follows from Theorem~IV.3.2 and Proposition~IV.3.14 in \cite{AKO2011}. See also \cite[Theorem~8.5.2]{Linckelmann2018} and note that there the terminology is different: Fusion systems in \cite{Linckelmann2018} are defined to be saturated fusion systems in our terminology.

\begin{theorem}\label{thm splitting field implies saturated}
Let $kG$, $b$, and $(P,e)$ be as in Definition~\ref{def block fusion system} and suppose that the $k$-algebra $kC_G(P)e$ is split, i.e., for every simple $kC_G(P)e$-module $V$ one has a $k$-algebra isomorphism $\End_{kC_G(P)e}(V)\cong k$. Then the fusion system $\calF_{(P,e)}(kGb)$ is saturated.
\end{theorem}

We are grateful to Radha Kessar who suggested the following example to us.

\begin{example}\label{ex Sylow axiom fails}
Let $p=2$, $k=\FF_2$, the field with $2$ elements, and $G:=D_{24}=(C_3\times C_4)\rtimes C_2$, the dihedral group with $24$ elements, with $C_2$ acting by inversion on $C_3\times C_4$. Let $g$ denote a generator of $C_3$. Then $b:=g+g^2$ is a block idempotent of $\FF_2 G$ and $(P,e):=(C_4,b)$ is a maximal $(\FF_2 G,b)$-Brauer pair. We have $\Aut_P(P)=\{1\}$, since $P$ is abelian and an easy computation shows that $\Aut_{\calF_{(P,e)}(\FF_2 Gb)}(P)\cong C_2$. Thus, the Sylow axiom does not hold for $\calF_{(P,e)}(\FF_2 Gb)$ and therefore the fusion system $\calF_{(P,e)}(\FF_2 Gb)$ is not saturated.
\end{example}


\section{Extension of scalars}\label{sec extension of scalars}

Throughout this section $L/K$ denotes a finite Galois extension of fields of characteristic $p>0$ and $\Gamma$ denotes its Galois group. Moreover, $G$ denotes a finite group.

$\Gamma$ acts via $K$-algebra automorphisms on the group algebra $LG$ and also on $Z(LG)$ by applying $\gamma\in\Gamma$ to the the coefficients of an element in $LG$. Thus, $\Gamma$ permutes the block idempotents of $LG$ and fixes the block idempotents $b$ of $KG$. Since $\Br_P\colon (LG)^P\to LC_G(P)$ commutes with the $\Gamma$-action, Theorem~\ref{thm maximal Brauer pairs} implies that any $\Gamma$-conjugate of $b$ has the same defect groups as $b$. We denote by $\Gamma_b$ the stabilizer of $b$ in $\Gamma$ and set
\begin{equation*}
  \btilde:=\sum_{\gamma\in[\Gamma/\Gamma_b]} \lexp{\gamma}{b}\,.
\end{equation*}
Clearly, $\btilde$ is an idempotent in $(Z(LG))^\Gamma=Z(KG)$. More precisely one has the following:

\begin{proposition}\label{prop Gamma and blocks}
{\rm (a)} Let $b$ be a block idempotent of $LG$. Then $\btilde:=\sum_{\gamma\in[\Gamma/\Gamma_b]} \lexp{\gamma}{b}$ is a block idempotent of $KG$.

\smallskip
{\rm (b)} The map $b\mapsto \btilde$ induces a bijection between the set of $\Gamma$-orbits of block idempotents of $LG$ and the set of block idempotents of $KG$.

\smallskip
{\rm (c)} If $b$ is a block idempotent of $LG$ and $\btilde$ is the block idempotent of $KG$ associated to it as in (a) then $b$ and $\btilde$ have the same defect groups.
\end{proposition}

\begin{proof}
(a) By definition, $\btilde$ is the sum of the distinct $\Gamma$-conjugates of $b$, thus an idempotent of $Z(KG)$. To see that $\btilde$ is primitive in $Z(KG)$, assume that $\btilde=c_1+c_2$ for non-zero orthogonal idempotents $c_1,c_2\in Z(KG)$ and let $I_1$ and $I_2$ denote the set of  primitive idempotents of $Z(LG)$ that occur in a primitive decomposition of $c_1$ and $c_2$ in $Z(LG)$, respectively. Then $I_1$ and $I_2$ are disjoint and $\Gamma$-stable. On the other hand $I_1\cup I_2$ is the single $\Gamma$-orbit of $b$. This is a contradiction.

\smallskip
(b) This is immediate from (a).

\smallskip
(c) Let $P$ be a defect group of $\btilde$. By Theorem~\ref{thm maximal Brauer pairs}, one has $\Br_P(\btilde)\neq 0$ in $KC_G(P)\subseteq LC_G(P)$. Thus $0\neq \Br_P(\btilde)=\sum_{\gamma\in[\Gamma/\Gamma_b]} \Br_P(\lexp{\gamma}{b})$ implies that some $\Gamma$-conjugate of $b$, and therefore also $b$, has a defect group $Q$ containing $P$. Thus, $0\neq\Br_Q(b)=\Br_Q(b\btilde)=\Br_Q(b)\Br_Q(\btilde)$, which implies that $\Br_Q(\btilde)\neq 0$ and therefore $|Q|\le|P|$. This implies $P=Q$.
\end{proof}

Note that $\Gamma$ acts on $\BP(LG)$ via 
\begin{equation}\label{eqn Gamma on BP}
  \lexp{\gamma}{(P,e)}=(P,\lexp{\gamma}{e})\,, 
\end{equation}
for $\gamma\in \Gamma$ and $(P,e)\in\BP(LG)$. Note that this action commutes with the $G$-action on $\BP(LG)$ so that we obtain an action of $\Gamma\times G$ on $\BP(LG)$. Moreover, since $\Br_P$ commutes with the action of $\Gamma$ and since the $G$-action on $LG$ commutes with the $\Gamma$-action on $LG$, $\Gamma\times G$ acts via poset isomorphisms on $\BP(LG)$. Thus, if $b$ is a block idempotent of $LG$ and $\gamma\in \Gamma$, the $G$-posets $\BP(LGb)$ and $\BP(LG\lexp{\gamma}{b})$ are isomorphic via (\ref{eqn Gamma on BP}) and $\Gamma_b\times G$ acts via poset automorphisms on $\BP(LGb)$.

\smallskip
In the next proposition we write $\le_K$ and $\le _L$ for the poset structures of $\BP(KG)$ and $\BP(LG)$, respectively. They are related as follows. 

\begin{proposition}\label{prop Gamma and le}
For $(Q,f),(P,e)\in\BP(LG)$ with $Q\le P$, the following are equivalent:

\smallskip
{\rm (i)} One has $(Q,\ftilde)\le_K (P,\etilde)$ in $\BP(KG)$.

\smallskip
{\rm (ii)} There exists $\gamma\in\Gamma$ such that $(Q,f)\le_L\lexp{\gamma}(P,e)$ in $\BP(LG)$.
\end{proposition}

\begin{proof}
Assume first that (i) holds and let $i$ be a primitive idempotent of $(KG)^P$ such that $\Br_P(i)\etilde=\Br_P(i)\neq 0$. Then, by definition also $\Br_Q(i)\ftilde=\Br_Q(i)\neq 0$. Let $J$ be a primitive decomposition of $i$ in $(LG)^P$. Since $\Br_P(i)\etilde\neq 0$, there exists $j\in J$ such that $\Br_P(j)\etilde\neq 0$. Thus, there exists $\gamma\in\Gamma$ such that $\Br_P(j)\lexp{\gamma}{e}\neq 0$. Since $\Br_P(j)$ is primitive in $LC_G(P)$, we have $\Br_P(j)\lexp{\gamma}{e}=\Br_P(j)$. Let $f'$ be the block idempotent of $LC_G(Q)$ such that $(Q,f')\le_L(P,\lexp{\gamma}{e})=\lexp{\gamma}(P,e)$. Then, by Theorem~\ref{thm (Q,f)le(P,e) equivalences} also $\Br_Q(j)f'=\Br_Q(j)\neq 0$. Thus $\Br_Q(j)f'\ftilde=\Br_Q(j)\Br_Q(i)f'\ftilde = \Br_Q(j)f' \Br_Q(i)\ftilde=\Br_Q(j)\Br_Q(i)=\Br_Q(j)\neq 0$ which implies that $f'\ftilde\neq 0$. This implies $f'=\lexp{\delta}{f}$ for some $\delta\in\Gamma$. Thus $\lexp{\delta}(Q,f)\le_L \lexp{\gamma}{(P,e)}$ and (ii) holds after applying $\delta^{-1}$.

\smallskip
Next assume that $\gamma\in\Gamma$ with $(Q,f)\le_L\lexp{\gamma}{(P,e)}$. By Theorem~\ref{thm Brauer pairs}(a) there exists a block idempotent $f_1$ of $LC_G(Q)$ such that $(Q,\ftilde_1)\le_K (P,\etilde)$. Since we already proved that (i) implies (ii), there exists $\delta\in \Gamma$ such that $(Q,f_1)\le_L\lexp{\delta}(P,e)$. Thus we have $(Q,\lexp{\gamma^{-1}}{f})\le_L (P,e)$ and also $(Q,\lexp{\delta^{-1}}{f})\le_L (P,e)$. The uniqueness part of Theorem~\ref{thm (Q,f)le(P,e) equivalences}(a) now implies that $f$ and $f_1$ are $\Gamma$-conjugate. Thus $\ftilde=\ftilde_1$ and $(Q,\ftilde)\le_K(P,\etilde)$.
\end{proof}

The following corollaries are now immediate from Proposition~\ref{prop Gamma and le}.

\begin{corollary}\label{cor BP(LG) and BP(KG)}
The map 
\begin{equation*}
  \BP(LG)\to \BP(KG)\,,\quad (P,e)\mapsto(P,\etilde)\,,
\end{equation*}
is a surjective morphism of $G$-posets, which restricts to a surjective morphism of $G$-posets $\BP(LGb)\to\BP(KG\btilde)$ for every block idempotent $b$ of $LG$.
\end{corollary}

\begin{corollary}\label{cor F and Ftilde}
Let $b$ be a block idempotent of $LG$ and let $(P,e)\in\BP(LGb)$ be a maximal $LGb$-Brauer pair. Then $(P,\etilde)\in\BP(KG\btilde)$ is a maximal $(KG\btilde)$-Brauer pair and one obtains an inclusion of fusion systems
\begin{equation*}
  \calF_{(P,e)}(LGb)\to \calF_{(P,\etilde)}(KG\btilde)
\end{equation*}
which is the identity on objects and on morphisms.
\end{corollary}


\section{The Main Theorem}\label{sec main theorem}

We keep $p$, $G$, $L/K$, and $\Gamma$ as introduced at the beginning of Section~\ref{sec extension of scalars}. 
Moreover we fix a block idempotent $b$ of $LG$ and denote by $\Gamma_b$ the stabilizer of $b$ in $\Gamma$. We fix a maximal $LGb$-Brauer pair $(P,e)\in\BP(LGb)$. For every $Q\le P$, let $e_Q$ denote the unique block idempotent of $LC_G(Q)$ such that $(Q,e_Q)\le (P,e)$ in $\BP(LG)$. By Proposition~\ref{prop Gamma and le}, one has $(Q,\widetilde{e_Q})\le(P,\etilde)$ so that $\widetilde{e_Q}=\etilde_Q$. This allows us to use the notation $\etilde_Q$ for both purposes.
Recall that $\Gamma\times G$ acts on $\BP(LG)$  and $\Gamma_b\times G$ acts on $\BP(LGb)$ via poset isomorphisms. Note that for any $(Q,f)\in\BP(LGb)$ one has $\Gamma_{(Q,f)}=\Gamma_f$. For the stabilizer in $G$ of a $KG$-Brauer pair or $LG$-Brauer pair $(Q,f)$ we will write $N_G(Q,f)$.

\bigskip
Let $p_1\colon G\times\Gamma\to G$ and $p_2\colon G\times \Gamma\to \Gamma$ denote the projection maps. For any subgroup $X$ of $G\times \Gamma$, we set $k_1(X):=\{g\in G\mid (g,1)\in X\}$ and $k_2(X):=\{\gamma\in\Gamma\mid (1,\gamma)\in X\}$. As explained in \cite[p.~24]{Bouc2010a}, one has
\begin{equation}\label{eqn q(X)}
  k_1(X)\trianglelefteq p_1(X)\le G \quad\text{and}\quad k_2(X)\trianglelefteq p_2(X)\le \Gamma\quad \text{with}\quad p_1(X)/k_1(X)\cong p_2(X)/k_2(X)
\end{equation}
via $g k_1(X) \leftrightarrow \gamma k_2(X)$ if and only if $(g,\gamma)\in X$. 

\smallskip
We denote by $K(b)$ and $K(e)$ the subfields of $L$ obtained by adjoining the coefficients of the block idempotents $b\in LG$ and $e\in LC_G(P)$. Thus, $K(b)$ is the fixed field of $\Gamma_b$ in $L$ and $K(e)$ is the fixed field of $\Gamma_e$ in $L$.

\begin{proposition}\label{prop Goursat invariants}
Let $b$ be a block idempotent of $LG$.

\smallskip
{\rm (a)} For any $(R,e_R)\le (Q,e_Q)$ in $\BP(LGb)$ one has $\Gamma_e=\Gamma_{(P,e)}\le \Gamma_{(Q,e_Q)}\le \Gamma_{(R,e_R)}\le \Gamma_{(\{1\},b)}=\Gamma_b$. In particular, $K(b)\subseteq K(e)$.

\smallskip
{\rm (b)} Let $X:=\stab_{G\times\Gamma}(P,e)$ be the stabilizer of the maximal $LGb$-Brauer pair $(P,e)$. One has
\begin{equation*}
  k_1(X)=N_G(P,e)\,,\quad p_1(X)=N_G(P,\etilde)\,,\quad k_2(X)=\Gamma_e\,,\quad\text{and}\quad 
  p_2(X)=\Gamma_b\,.
\end{equation*}

\smallskip
{\rm (c)} One has $N_G(P,e)\trianglelefteq N_G(P,\etilde)$ and $N_G(P,\etilde)/N_G(P,e)\cong \Gamma_b/\Gamma_e$. Moreover, $K(e)/K(b)$ is a Galois extension with cyclic Galois group isomorphic to $N_G(P,\etilde)/N_G(P,e)$.
\end{proposition}

\begin{proof}
(a) It suffices to show that $\Gamma_{(Q,e_Q)}\le\Gamma_{(R,e_R)}$. Let $\gamma\in\Gamma_{(Q,e_Q)}$. Then $\lexp{\gamma}{(R,e_R)}\le_L\lexp{\gamma}{(Q,e_Q)}=(Q,\lexp{\gamma}{e_Q})=(Q,e_Q)$. The uniqueness part of Theorem~\ref{thm Brauer pairs}(a) implies that $\lexp{\gamma}{e_R}=e_R$. Thus, $\gamma\in\Gamma_{(R,e_R)}$.

\smallskip
(b) The first equation is clear from the definition of $k_1(X)$. For the proof of the second equation, let $g\in p_1(X)$. Then there exists $\gamma\in\Gamma$ with $(P,e)=\lexp{(g,\gamma)}{(P,e)}=(\lexp{g}{P},\lexp{g\gamma}{e})$. From $\lexp{g\gamma}{e}=e$ it follows that $\lexp{g}{\etilde}=\etilde$. Thus $\lexp{g}(P,\etilde)=(P,\etilde)$ and $g\in N_G(P,\etilde)$. Conversely, if $g\in N_G(P,\etilde)$ then $\lexp{g}{\etilde}=\etilde$ which implies that there exists $\gamma\in\Gamma$ with $\lexp{g}{e}=\lexp{\gamma}{e}$. Thus, $\lexp{(g,\gamma^{-1})}{(P,e)} =(P,e)$ and $g\in p_1(X)$. The third equation follows immediately from the definition of $k_2(X)$.  For the proof of the fourth equation let $\gamma\in p_2(X)$. Then there exists $g\in G$ with $\lexp{(g,\gamma)}{(P,e)}=(P,e)$. Since $(\{1\},b)\le (P,e)$, this implies $\lexp{(g,\gamma)}{(\{1\},b)}\le\lexp{(g,\gamma)}{(P,e)}=(P,e)$. The uniqueness part in Theorem~\ref{thm Brauer pairs}(a) implies that $\lexp{(g,\gamma)}{(\{1\},b)}=(1,b)$ and that $\gamma\in\Gamma_b$. Conversely, assume that $\gamma\in \Gamma_b$. Then $(\{1\},b)\le (P,e)$ implies $(\{1\},b)=\lexp{(1,\gamma)}{(\{1\},b)}\le\lexp{(1,\gamma)}{(P,e)}=(P,\lexp{\gamma}{e})$. This implies that both $(P,e)$ and $\lexp{\gamma}{(P,e)}$ are maximal $LGb$-Brauer pairs. By Theorem~\ref{thm maximal Brauer pairs}(a), there exists $g\in G$ such that $\lexp{g}{(P,\lexp{\gamma}{e})}=(P,e)$. Thus $(g,\gamma)\in X$ and $\gamma\in p_2(X)$.

\smallskip
(c) The assertions of the first sentence follow from Part~(b) and (\ref{eqn q(X)}). For the second statement it suffices to show that $\Gamma_b/\Gamma_e$ is cyclic. Note that the coefficients of $e\in LC_G(P)$ generate a finite field extension of the prime field $\FF_p$ in $L$, which we denote by $\FF_p(e)$. Since $\Gamma_e\trianglelefteq\Gamma_b$, we have a Galois extension $K(e)/K(b)$ with Galois group $\Delta\cong \Gamma_b/\Gamma_e$. Now, restriction from $K(e)$ to $\FF_p(e)$ is an injective group homomorphism from $\Delta$ to the cyclic Galois group $\Gal(\FF_p(e)/\FF_p)$. In fact, if $\delta\in\Delta$ restricts to the identity on $\FF_p(e)$, then it is the identity on $\FF_p(e)$ and on $K$, thus on $K(e)$. This completes the proof of Part~(c).
\end{proof}

Next we give a more precise picture of the inclusion of fusion systems from Corollary~\ref{cor F and Ftilde}. In the following theorem the term $\langle \calF,\sigma\rangle$ denotes the fusion system {\em generated} by $\calF$ and $\sigma$, i.e., the intersection of all fusion systems over $P$ that contain $\calF$ and $\sigma$.

\begin{theorem}\label{thm A}
Let $L/K$ be a finite Galois extension of fields of characteristic $p>0$ with Galois group $\Gamma$, let $b$ be a block idempotent of $LG$, and let $(P,e)$ be a maximal $LGb$-Brauer pair. Set $\calF:=\calF_{(P,e)}(LGb)$ and $\calFtilde:=\calF_{(P,\etilde)}(KG\btilde)$. Let $g_0\in N_G(P,e)$ be such that $g_0N_G(P,\etilde)$ generates $N_G(P,e)/N_G(P,\etilde)$ (see Proposition~\ref{prop Goursat invariants}(c)) and set $\sigma:=c_{g_0}\in\Aut(P)$. Then $\calFtilde=\langle \calF, \sigma \rangle$. 

More precisely, $\sigma\in\Aut_{\calFtilde}(P)$ and, for any subgroups $Q$ and $R$ of $P$ and any $\varphi\in\Hom_{\calFtilde}(Q,R)$, there exist $i\in\ZZ$, $\psi\in\Hom_{\calF}(Q,\sigma^{-i}(R))$ and $\psi'\in\Hom_{\calF}(\sigma^i(Q),R)$ with $\varphi=\sigma^i|_{\sigma^{-i}(R)}\circ \psi=\psi'\circ\sigma^i|_Q$.
\end{theorem}

\begin{proof}
Since $g_0\in N_G(P,\etilde)$, we have $\sigma=c_{g_0}\in\Aut_{\calFtilde}(P)$. It follows that $\langle \calF,\sigma\rangle\subseteq \calFtilde$. In order to prove the reverse inclusion, let $Q$ and $R$ be subgroups of $P$ and let $\varphi\in\Hom_{\calFtilde}(Q,R)$. Then there exists $g\in G$ such that $\varphi=c_g\colon Q\to R$ and $\lexp{g}{(Q,\etilde_Q)}\le_K (R,\etilde_R)$. By Proposition~\ref{prop Gamma and le} there exists $\gamma\in\Gamma$ such that $\lexp{g}{(Q,e_Q)}\le_L (R,\lexp{\gamma}{e_R})$. Since $(\{1\},b)=\lexp{g}{(\{1\},b)} \le_L \lexp{g}{(Q,e_Q)}\le_L(R,\lexp{\gamma}{e_R})$ and also $(\{1\},\lexp{\gamma}{b})\le_L(R,\lexp{\gamma}{e_R})$, Theorem~\ref{thm Brauer pairs}(a) implies $(\{1\},b)=(\{1\},\lexp{\gamma}{b})$ so that $\gamma\in\Gamma_b$. Thus, both $(P,e)$ and $(P,\lexp{\gamma}{e})$ are maximal $LGb$-Brauer pairs. Theorem~\ref{thm maximal Brauer pairs}(a) implies that there exists $h\in G$ such that $\lexp{h}{(P,e)}=(P,\lexp{\gamma}{e})$ and we obtain $(P,e)=\lexp{h^{-1}}{(P,\lexp{\gamma}{e})}\ge_L \lexp{h^{-1}}{(R,\lexp{\gamma}{e_R})} = (\lexp{h^{-1}}{R},\lexp{h^{-1}\gamma}{e_R})$. Again, Theorem~\ref{thm Brauer pairs}(a) implies that $\lexp{h^{-1}\gamma}{e_R}=e_{h^{-1}Rh}$ and therefore $\lexp{h^{-1}g}{(Q,e_Q)}\le_L\lexp{h^{-1}}{(R,\lexp{\gamma}{e_R})} = (\lexp{h^{-1}}{R},e_{h^{-1}Rh})$. This in turn implies that the homomorphism $\alpha:= c_{h^{-1}g}\colon Q\to \lexp{h^{-1}}{R}$ belongs to $\Hom_{\calF}(Q,\lexp{h^{-1}}{R})$ and that the homomorphism $\varphi=c_g\colon Q\to R$ factors as
\begin{equation}\label{eqn phi factorization}
  \varphi=c_h\circ \alpha\colon Q\to \lexp{h^{-1}}{R}\to R\,.
\end{equation}
Since $\lexp{h}{(P,e)}=(P,\lexp{\gamma}{e})$, we obtain $h\in N_G(P,\etilde)$ and can write $h=g_0^ix$ for some $i\in\ZZ$ and $x\in N_G(P,e)$. This implies that the map $c_h\colon P\to P$ factors as $c_h=\sigma^i\circ\beta\colon P\to P$ where $\sigma^i=c_{g_0^i}\colon P\to P$ and $\beta:=c_x\in \Aut_\calF(P)$, since $x\in N_G(P,e)$. Restriction to $\lexp{h^{-1}}{R}$ yields the factorization
\begin{equation*}
  c_h|_{h^{-1}Rh} = \sigma^i|_{\beta(h^{-1}Rh)}\circ \beta|_{h^{-1}Rh}\colon \lexp{h^{-1}}{R} \to \beta(\lexp{h^{-1}}{R})\to R
\end{equation*}
with $\beta(\lexp{h^{-1}}{R})=\sigma^{-i}(R)$ and $\beta|_{h^{-1}Rh}\in\Hom_{\calF}(\lexp{h^{-1}}{R}, R)$. Setting $\psi:=\beta|_{h^{-1}Rh}\circ\alpha\colon Q\to \sigma^{-1}(R)$ and using (\ref{eqn phi factorization}) we obtain the desired factorization of $\varphi$. This also implies the inclusion $\calFtilde\subseteq\langle \calF,\sigma\rangle$.

\smallskip
In order to find $\psi'$ with the desired property we use the elements $g$, $h$, $x$, and $i$ from the first part of the proof and  note that
\begin{equation*}
  (P,e) = \lexp{\gamma^{-1}h}({P,e)}\ge_L \lexp{\gamma^{-1}}{(\lexp{h}{Q},\lexp{h}{e_Q})} = (\lexp{h}{Q},\lexp{\gamma^{-1}h}{e_Q})\,,
\end{equation*}
which implies that $\lexp{\gamma^{-1}h}{e_Q}=e_{hQh^{-1}}$. Thus,
\begin{equation*}
  \lexp{gh^{-1}}{(\lexp{h}{Q},e_{hQh^{-1}})} = \lexp{gh^{-1}}{(\lexp{h}{Q},\lexp{\gamma^{-1}h}{e_Q})} = (\lexp{g}{Q},\lexp{g\gamma^{-1}}{e_Q}) \le_L(R,e_R)\,,
\end{equation*}
which implies that $\alpha':=c_{gh^{-1}}\colon \lexp{h}{Q}\to R$ belongs to $\Hom_\calF(\lexp{h}{Q},R)$. Thus, $\varphi$ can be factored as
\begin{equation}\label{eqn phi factorization 2}
  \varphi= c_g=c_{gh^{-1}}\circ c_h=\alpha'\circ c_h\colon Q\to \lexp{h}{Q}\to R\,.
\end{equation} 
We can rewrite $h=g_0^ix=x'g_0^i$ for some $x'\in N_G(P,e)$ and obtain an element $\beta'\in\Aut_{\calF}(P)$ together with a factorization $c_h=\beta'\circ \sigma^i\colon P\to P$. Restricting this equation to $Q$ yields a factorization
\begin{equation*}
  c_h= \beta'|_{\sigma^i(Q)}\circ \sigma^i|_Q\colon Q\to \sigma^i(Q)\to \lexp{h}{Q}\,.
\end{equation*}
Setting $\psi':=\alpha'\circ\beta'|_{\sigma^i(Q)}\in\Hom_{\calF}(\sigma^i(Q),R)$, the factorization in (\ref{eqn phi factorization 2}) can now be expressed as $\varphi = \psi'\circ\sigma^i|_Q$ as claimed.
\end{proof}


\section{Consequences of the Main Theorem}\label{consequences}

In this section we prove several consequences of Theorem~\ref{thm A}.


Recall that if $\calF$ is a fusion system over a $p$-group $P$, a subgroup $Q$ of $P$ is called {$\calF$-centric} if $C_P(R)=Z(R)$ for all subgroups $R$ of $P$ which are $\calF$-isomorphic to $Q$.

\begin{proposition}\label{prop calF calFtilde centralized}
Let $L/K$, $b$, $(P,e)$ and $\calF\subseteq\calFtilde$ be as in Theorem~\ref{thm A}.

\smallskip
{\rm (a)} A subgroup $Q$ of $P$ is fully $\calF$-centralized if and only if it is fully $\calFtilde$-centralized.

\smallskip
{\rm (b)} A subgroup $Q$ of $P$ is fully $\calF$-normalized if and only if it is fully $\calFtilde$-normalized.

\smallskip
{\rm (c)} A subgroup $Q$ of $P$ is $\calF$-centric if and only it is $\calFtilde$-centric.
\end{proposition}

\begin{proof}
The \lq if\rq-parts follow immediately from the fact that the $\calF$-isomorphism class of $Q$ is a subset of the $\calFtilde$-isomorphism class of $Q$. For the forward implications note that by Theorem~\ref{thm A} two subgroups $Q$ and $Q'$ of $P$ are $\calFtilde$-isomorphic if and only if there exists a subgroup $Q''$ of $P$ such that $Q$ is $\calF$-isomorphic to $Q''$ and $Q'=\sigma^i(Q'')$ for some $i\in \ZZ$. Moreover, $\sigma^i(C_P(Q''))=C_P(\sigma^i(Q''))$, $\sigma^i(N_P(Q''))=N_P(\sigma^i(Q''))$, and $\sigma^i(Z(Q''))=Z(\sigma^i(Q''))$, since $\sigma^i$ is an automorphism of $P$. The result is now immediate.
\end{proof}

The following Theorem is known to experts. See for instance the part of the proof of \cite[Theorem~8.5.2]{Linckelmann2018} dealing with the extension axiom and note that it does not use any assumptions on the field of coefficients $k$. Below is a proof with a different approach, using Theorem~\ref{thm A}.

\begin{theorem}\label{thm B}
Let $k$ be a field of characteristic $p>0$ and let $c$ be a block idempotent of $kG$. Then the extension axiom holds for the fusion system of $kGc$, for any choice of maximal Brauer pair.
\end{theorem}

\begin{proof}
Let $(P,f)$ be a maximal $kGc$-Brauer pair. We apply Theorem~\ref{thm A} with $K=k$, a splitting field $L$ of $KC_G(P)f$ such that $L/K$ is a finite Galois extension with Galois group $\Gamma$, and to a block idempotent $b$ of $LG$ with $cb\neq 0$. Then $c=\btilde$. Moreover, there exists a maximal $LGb$-Brauer pair $(P,e)$ such that $ef=e$ and therefore $f=\etilde$. We aim to show that the fusion system $\calFtilde=\calF_{(P,\etilde)}(KG\btilde)$ satisfies the extension axiom. Note that by Theorem~\ref{thm splitting field implies saturated}, the extension axoim holds for $\calF=\calF_{(P,e)}(LGb)$, since $L$ is a splitting field of $LC_G(P)e$. Let $\varphi\in\Hom_{\calFtilde}(Q,P)$ be such that $\varphi(Q)$ is fully $\calFtilde$-normalized. By Theorem~\ref{thm A} we can factorize $\varphi=\sigma^i\circ \psi$ for some $\psi\in \Hom_{\calF}(Q,P)$. 
With $\varphi(Q)$ also $\psi(Q)=\sigma^{-i}(\varphi(Q))$ is fully $\calFtilde$-normalized, since they are $\calFtilde$-isomorphic and $N_P(\psi(Q))=\sigma^{-i}(N_P(\varphi(Q)))$. By Proposition~\ref{prop calF calFtilde centralized}(b), $\psi(Q)$ is fully $\calF$-normalized. Since $\calF$ satisfies the extension axiom, there exists $\psihat\in\Hom_{\calF}(N_\psi,P)$ such that $\psihat|_Q=\psi$. It follows that $\varphihat:=\sigma^i\circ \psihat\in\Hom_{\calFtilde}(N_\psi,P)$ extends $\varphi$. To finish the proof it suffices to show that $N_\varphi\subseteq N_\psi$. So let $x\in N_\varphi$. Then $x\in N_P(Q)$ and there exists $y\in N_P(\varphi(Q))$ with $ \varphi\circ c_x=c_y\circ\varphi\colon Q\myiso \varphi(Q)$. But this implies
\begin{equation*}
  \psi\circ c_x = \sigma^{-i}\circ \varphi\circ c_x = \sigma^{-i}\circ c_y\circ \varphi = c_{\sigma^{-i}(y)}\circ \sigma^{-i}\circ \varphi = c_{\sigma^{-i}(y)}\circ\psi\,,
\end{equation*}
with $\sigma^{-i}(y)\in \sigma^{-i}(N_P(\varphi(Q)))=N_P(\sigma^{-i}(\varphi(Q)))= N_P(\psi(Q))$. Thus, $N_\varphi\subseteq N_\psi$ and the proof is complete.
\end{proof}

\begin{theorem}\label{thm C}
Let $L/K$, $b$, $(P,e)$ and $\calF\subseteq\calFtilde$ be as in Theorem~\ref{thm A}. The fusion system $\calFtilde$ is saturated if and only if the fusion system $\calF$ is saturated and $p$ does not divide $[N_G(P,\etilde):N_G(P,e)]=[\Gamma_b:\Gamma_e]=[K(e):K(b)]$. In particular, if moreover $L$ is a splitting field for $LC_G(P)e$, then $\calFtilde$ is saturated if and only if $p$ does not divide $[N_G(P,\etilde):N_G(P,e)]=[\Gamma_b:\Gamma_e]=[K(e):K(b)]$.
\end{theorem}

\begin{proof}
Note that the map $N_G(P,e)\to\Aut_{\calF}(P)$, $g\mapsto c_g$ induces an isomorphism $N_G(P,e)/C_G(P)\to \Aut_{\calF}(P)$ which maps $PC_G(P)/C_G(P)$ to $\Aut_P(P)$. Thus, the Sylow axiom holds for $\calF$ if and only if $p\nmid[N_G(P,e):PC_G(P)]$. Similarly, the Sylow axiom holds for $\calFtilde$ if and only if $p\nmid[N_G(P,\etilde):PC_G(P)]$.
By Theorem~\ref{thm B} it suffices to show that the Sylow axiom holds for $\calFtilde$ if and only it holds for $\calF$ and $p\nmid[\Gamma_b:\Gamma_e]$. But, by Proposition~\ref{prop Goursat invariants}(c), one has $[\Gamma_b:\Gamma_e]=[N_G(P,\etilde):N_G(P,e)]=[K(e):K(b)]$ which implies the result.
\end{proof}

Next we will show that a weak form of Alperin's fusion theorem holds for arbitrary block fusion systems. 

\begin{definition}
Let $\calF$ be a fusion system over a $p$-group $P$. 
We say that {\em Alperin's weak fusion theorem holds for $\calF$} if $\calF=\langle \Aut_\calF(Q)\mid Q\in\calC\rangle$, where $\calC$ is the set  of subgroups of $P$ which are $\calF$-centric and fully $\calF$-normalized.
\end{definition}

\begin{theorem}\label{thm D}
Let $k$ be a field of characteristic $p$ and let $c$ be a block idempotent of $kG$. Then Alperin's weak fusion theorem holds for the fusion system of $kGc$, for any choice of maximal $kGc$-Brauer pair.
\end{theorem}

\begin{proof}
Set $K:=k$ and choose $L$, $b$, $(P,e)$ as in the proof of Theorem~\ref{thm B} with $c=\btilde$ and apply Theorem~\ref{thm A} to this situation with $\calF:=\calF_{(P,e)}(LGb)$ and $\calFtilde:=\calF_{(P,\etilde)}(KG\btilde)$. We need to show that Alperin's weak fusion theorem holds for $\calFtilde$. Since $\calF$ is saturated, Alperin's weak fusion theorem holds for $\calF$, see for instance \cite[Theorem~8.2.8]{Linckelmann2018}. Thus, $\calF=\langle \Aut_\calF(Q)\mid Q\in\calC\rangle$, where $\calC$ denotes the set of subgroups of $P$ which are $\calF$-centric and fully $\calF$-normalized. Moreover, by Proposition~\ref{prop calF calFtilde centralized}, $\calC$ is equal to the set $\calCtilde$ of subgroups of $P$ which are $\calFtilde$-centric and fully $\calFtilde$-normalized. Thus, by Theorem~\ref{thm A}, we have 
\begin{equation*}
  \calFtilde=\langle \calF,\sigma\rangle = \langle\{\Aut_\calF(Q)\mid Q\in\calC\}\cup\{\sigma\}\rangle \subseteq \langle \Aut_{\calFtilde}(Q)\mid Q\in\calC\rangle \subseteq \calFtilde\,.
\end{equation*}
But this implies $\calFtilde=\langle \Aut_{\calFtilde}(Q)\mid Q\in\calC\rangle=\langle \Aut_{\calFtilde}(Q)\mid Q\in\calCtilde\rangle$, which  means that Alperin's weak fusion theorem holds for $\calFtilde$.
\end{proof}



\begin{thebibliography}{0000}

\bibitem[AKO11]{AKO2011} 
   {\sc M.~Aschbacher, R.~Kessar, B.~Oliver:} 
   Fusion systems in Algebra and topology.
   London Mathematical Society Lecture Note Series, 391. Cambridge University Press, Cambridge, 2011.
\bibitem[B10]{Bouc2010a}{\sc S.~Bouc:} 
   Biset functors for finite groups.
   Lecture Notes in Mathematics, 1990. Springer-Verlag, Berlin, 2010.
\bibitem[L18]{Linckelmann2018} 
   {\sc M.~Linckelmann:} 
   The block theory of finite group algebras. Vol.~II. 
   London Mathematical Society Student Texts, 92. Cambridge University Press, Cambridge, 2018.


\end{thebibliography}
\end{document}